\documentclass[12pt,a4paper,oneside,english]{amsart}

\usepackage[T1]{fontenc}
\usepackage[latin9]{inputenc}
\usepackage{verbatim}
\usepackage{amsthm}
\usepackage{amstext}
\usepackage{amssymb}
\usepackage{diagrams}
\usepackage{babel}
\usepackage{pinlabel}

\numberwithin{equation}{section}
\numberwithin{figure}{section}
\theoremstyle{remark}
\newtheorem*{acknowledgement*}{Acknowledgement}
\theoremstyle{plain}
\newtheorem{thm}{Theorem}
\theoremstyle{definition}
\newtheorem{defn}[thm]{Definition}
\theoremstyle{plain}
\newtheorem{lem}[thm]{Lemma}
\theoremstyle{definition}
\newtheorem{example}[thm]{Example}
\theoremstyle{remark}
\newtheorem{rem}[thm]{Remark}
\theoremstyle{plain}
\newtheorem{cor}[thm]{Corollary}
\theoremstyle{plain}
\newtheorem{prop}[thm]{Proposition}

\newcommand{\R}{\mathbb{R}}
\newcommand{\Energy}{\mathcal{E}}
\newcommand{\core}{{\bf Core}}
\newcommand{\Core}{{\bf core}}
\newcommand{\im}{{\operatorname{Im}}}
\newcommand{\Cot}{{T}^*}
\newcommand{\id}{\operatorname{id}}
\newcommand{\graph}{\operatorname{gr}}

\begin{document}

\title{Symplectic Microgeometry II:\\ Generating functions\\}

\author{Alberto S. Cattaneo, Benoit Dherin and Alan Weinstein}

\begin{abstract}
We adapt the notion of generating functions for lagrangian submanifolds to symplectic microgeometry. We show that a symplectic micromorphism always admits a global generating function. As an application, we describe hamiltonian flows as special symplectic micromorphisms whose local generating functions are the solutions of Hamilton-Jacobi equations. We obtain a purely categorical formulation of the temporal evolution in classical mechanics.
\end{abstract}

\maketitle
\tableofcontents{}

\section{Introduction}

This article is a continuation of \cite{SM_Microfolds}, in which we introduced the notion of symplectic microfolds and symplectic micromorphisms between them. Recall that a symplectic microfold is essentially the same thing as a germ $[M,A]$ of symplectic manifolds $M$ around a lagrangian submanifold $A$, called the core. A symplectic micromorphism from $[M,A]$ to $[N,B]$ is a germ around the graph of a smooth map $\phi:B\rightarrow A$, the core map, of a canonical relation\footnote{The opposite symplectic manifold $\overline{M}$ of a symplectic manifold $(M,\omega_{M})$ is the manifold $M$ endowed with the opposite symplectic form $-\omega_{M}$.} $V\subset\overline{M}\times N$ containing the graph of $\phi$ and satisfying a certain transversality condition; we review the definitions in Section \ref{subsec-definitions} below.

In contrast with their macroworld counterparts (i.e. canonical relations), symplectic micromorphisms always compose well, forming thus, with the symplectic microfolds as objects, an honest symmetric monoidal category: the microsymplectic category, which can be thought as the appropriate ``microworld'' analog to the symplectic ``category'' of symplectic manifolds and canonical relations \cite{Weinstein1977}. This makes the microsymplectic category a natural setting for questions related to the functorial behavior of quantization schemes in symplectic geometry. 

By ``quantization scheme'', we mean here any well-defined correspondence between a certain type of geometric structures in the symplectic realm (such as Poisson structures) and a certain type of algebraic structures in the realm of analysis (such as $C^{*}$-algebras).  Two important ingredients involved in these quantization questions have been the generating functions and the Fourier integral operators associated with a given lagrangian submanifold (see \cite{BW1997,GS1977}). 

In this paper, we extend the notion of generating functions to the microworld. Our first main result is that \textit{any} symplectic micromorphism admits a \textit{global} generating function. This is the best possible case when it comes to quantization via Fourier integral operators (as will be shown in a sequel \cite{SM_Quantization}). 

A second result (Theorem \ref{thm: normal form}) states that the underlying lagrangian submicrofold of a symplectic micromorphism can always be decomposed into a fibration by graphs of actual smooth map germs. The main example here is the cotangent lift $\Cot\phi:\Cot A\rightarrow\Cot B$ of a smooth map $\phi:B\rightarrow A$, whose underlying lagrangian submicrofold is the germ of the lagrangian submanifold
\[
	\Big\{\Big(\big(p_{1},\,\phi(x_{2})\big),\,\big((\Cot_{x_{2}}\phi)p_{1},\, x_{2}\big)\Big):\,(p_{1},x_{2})\in\phi^{*}(\Cot A)\Big\}
\]
around the graph of $\phi$. When $\phi$ is not a diffeomorphism, $\Cot\phi$ is not the graph of a symplectomorphism, although the cotangent maps $\Cot_{x_{2}}\phi:\Cot_{\phi(x_{2})}A\rightarrow\Cot_{x_{2}}B$ to $\phi$ at each $x_{2}\in B$ \textit{are} actual maps. This gives us a decomposition
\[
	\Cot\phi=\bigcup_{x_{2}\in B}\graph\Cot_{x_{2}}\phi
\]
of $\Cot\phi$ associated with the lagrangian fibration of $\Cot B$ by its cotangent fibers. More generally, a similar decomposition holds for general symplectic micromorphisms from $[M,A]$ to $[N,B]$, which is uniquely associated with the data of a lagrangian fibration of $[N,B]$. This special geometry of the underlying lagrangian submicrofolds may be of help in questions related to the continuity of their corresponding Fourier integral operators. Namely, in various instances (\cite{Eskin1970},\cite{Hormander1971},\cite{Ruz2000}), the continuity in $L^2$-spaces of some classes of FIOs has been related to their wave-fronts being local graphs. From a different perspective, this decomposition is reminiscent of the notion of co-morphisms of Chen and Liu introduced in \cite{CL2007} in the context of Lie groupoids. 

We derive the existence of  global generating functions for symplectic micromorphisms from a more general result: the \textit{equivalence theorem} for clean lagrangian submicrofolds (Theorem \ref{thm: equivalence theorem}). It states that the data of a germ $[L,C]$ of a lagrangian submanifold $L\subset\Cot A$ that intersects the zero section \textit{cleanly} in $C$ is equivalent to the data $(K,f)$ of a symplectomorphism germ \footnote{Given a vector bundle $E\rightarrow A$,  when it is clear from the context, we write $E$ to denote the microfold $[E,Z_{A}]$, where $Z_{A}$ is the zero section of $E$. Moreover, we also identify $Z_{A}$ with $A$ and the submanifolds $C\subset A$ with the corresponding submanifolds of $Z_{A}$, yielding the notation $[E,C]$.}  $K:\Cot N^{*}C\rightarrow[\Cot A,N^{*}C]$ (of the type prescribed by the lagrangian embedding theorem \cite{Weinstein71} applied to the conormal bundle $N^{*}C\subset\Cot A$ and satisfying an extra condition) together with a function germ $f:N^{*}C\rightarrow\R$ around the zero section (and vanishing on it, as well as its differential). Actually, we prove that for each $K$ there exists a unique $f$ such that $L=K(\im\, df)$. This defines the global generating function $f$ of the clean lagrangian submicrofold $[L,C]$ associated with the symplectomorphism germ $K$. 

It turns out that the symplectic micromorphisms from $\Cot A$ to $\Cot B$ with core map $\phi$ are nothing but the \textit{clean} lagrangian submicrofolds in $\Cot(A\times B)$ with core $\graph\phi$. While giving another (and simpler) characterization of symplectic micromorphisms, the equivalence theorem also tells us that they admit global generating functions (associated with special symplectomorphism germs $K$ as before). 

Another application of the equivalence theorem comes from considering the restriction of a symplectic micromorphism to a local chart. The restriction is again a symplectic micromorphism, but now we have a canonical symplectomorphism germ $K$ coming from the affine structure of the local chart. This allows us to define the \textit{local} generating function of the symplectic micromorphism in the local chart as the global generating function of its restriction.

We conclude this paper by associating with any hamiltonian flow on a cotangent bundle a symplectic micromorphism, the evolution micromorphism, which encodes the dynamics for asymptotically short times. We show that the local generating functions of the evolution micromorphism coincide with the solutions of the Hamilton-Jacobi equation for the flow in local charts.   Moreover, by considering a particular monoid object in the microsymplectic category, the energy monoid $\Cot\Energy$, we show that the $\Cot\Energy$-modules are essentially the same thing as germs of hamiltonian flows with time-independent hamiltonian. This gives us a categorical formulation for the temporal evolution in classical mechanics. Finally, we briefly describe how symmetries in classical mechanics can be formalized using the language of symplectic microgeometry. This approach to symmetry, which will be developed in full details somewhere else, is very close in spirit to the work of Benenti on the Hamilton-Jacobi equation for an hamiltonian action (\cite{Benenti1983}).
 
\begin{acknowledgement*}
A.S.C. acknowledges partial support from SNF Grant 200020\_131813/1.
B.D. thanks Pedro de M. Rios for useful discussions on generating
families and acknowledges partial support from NWO Grant 613.000.602 carried out at Utrecht University and from FAPESP grant 2010/15069-8 hosted by the ICMC of Sao Paulo University at Sao Carlos. A.W. acknowledges partial support from NSF grant DMS-0707137 and the hospitality of the Institut Math\'ematique de Jussieu.
\end{acknowledgement*}

\section{Transversality and deformations\label{sec:Transversality-and-deformations}}

In this section, we start by recalling some basic definitions concerning the geometry of manifold germs or microfolds (see \cite{SM_Microfolds,Weinstein71} for more details). Then we focus on the geometry of lagrangian submanifolds in cotangent bundles around their intersection with the zero section. This geometry is captured by the notion of a lagrangian submicrofold $[L,C]$ in $[\Cot A,Z_{A}]$, that is, a germ $[L]$ of lagrangian submanifolds $L\subset\Cot A$ around $C\subset A$. 

The main result of this section is the equivalence theorem (Theorem \ref{thm: equivalence theorem}) which states that the lagrangian submicrofolds $[L,C]$ that intersect the zero section \textit{cleanly} in $C$ coincide with certain deformations of the conormal microbundle $[N^{*}C,C]$. Moreover, the cleanliness assumption is enough to define a notion of global generating function for the lagrangian submicrofold.

\subsection{Definitions}
\label{subsec-definitions}

A \textbf{microfold} is an equivalence class $[G,A]$ of manifold pairs $(G,A)$, where $A$ is a closed submanifold of $G$; two pairs $(G_{1},A)$ and $(G_{2},A)$ are equivalent if there exists a third one $(G_{3},A)$ for which $G_{3}$ is simultaneously an open submanifold of both $G_{1}$ and $G_{2}$. 

The manifold $A$ is called the \textbf{core} of the microfold $[G,A]$. In other words, $[G,A]$ is a manifold germ around $A$. 

A morphism $[\Phi]:[G,A]\to [H,B]$  between microfolds is an equivalence class of smooth maps $\Phi:(G,A)\to
(H,B)$ between representatives, where two such maps are equivalent if there is a common neighborhood of $A$ on which they are both defined, and equal. The morphism is an isomorphism when there is a representative map which is a diffeomorphism. Any such isomorphism induces a diffeomorphism between the cores. 

A \textbf{microbundle} is a microfold $[E,Z_{A}]$ obtained from a vector bundle $E\rightarrow A$ by taking the germ of $E$ around the zero section $Z_{A}$. (Every microfold is diffeomorphic to a microbundle.) When clear from the context, we will write $E$ instead of $[E,Z_{A}]$. Throughout, we will use the canonical identifications
\begin{eqnarray*}
   T_{(0,x)}E        & = & E_{x}\oplus T_{x}A,\\
   T_{(0,x)}Z_{A} & = & \{0\}\oplus T_{x}A,\\
   V_{(0,x)}E       & = & E_{x}\oplus\{0\},
\end{eqnarray*}
where the \textbf{vertical} \textbf{bundle} $VE\rightarrow E$ of a vector bundle $\pi:E\rightarrow A$ is the subbundle of $TE$ whose fiber at $e\in E$ is the kernel of $T_{e}\pi$. If $C$ is a submanifold of $A$ and $E\rightarrow A$ a vector bundle, we will write $E_{|C}$ to denote the restriction of $E$ along $C$. 

A \textbf{submicrofold} $[H,B]$ of $[G,A]$ is a microfold for which there are representatives $(H,B)$ and $(G,A)$ such that $B$ is a submanifold of $A$ and $H$ is a submanifold of $G$, and $H\cap A=B$. We say that $[H,B]$ is a \textbf{clean} submicrofold if the intersection of $H$ with the core $A$ is clean (i.e. $TH\cap TA=TB$).

A \textbf{symplectic microfold} is a microfold $[M,A]$ such that $M$ is a symplectic manifold and $A$ is a lagrangian submanifold. 

A \textbf{lagrangian splitting} of a symplectic microfold $[M,A]$ is a lagrangian subbundle $K\rightarrow A$ of $TM|_A$ that is transverse to $A$.

A \textbf{lagrangian submicrofold} $[L,C]$ of $[M,A]$ is a submicrofold such that $L$ is lagrangian in $M$. 

A \textbf{symplectic micromorphism} from $[M,A]$ to $[N,B]$ is a lagrangian submicrofold 
\[
	[V,\graph \phi] \subset [\overline{M}\times N, A\times B],
\] 
whose core is the graph of a smooth map $\phi:B\rightarrow A$ and that satisfies the following {\it transversality condition} :
We require  $TV$ to be transverse along $\graph \phi$  to a (and hence any) lagrangian subbundle of the form 
\[
	(TA\oplus 0)\times (0\oplus K),
\]
where $K$ is a lagrangian splitting of $[N,B]$.

To distinguish symplectic micromorphisms from other lagrangian submicrofolds, we use the special notation $([V],\phi):[M,A]\rightarrow [N,B] $ instead of $[V,\graph\phi]$.

\begin{rem}
The definition of symplectic micromorphisms above differs from the one given in \cite[Def. 3.1]{SM_Microfolds} but is equivalent to it as stated in \cite[Cor 3.1]{SM_Microfolds}. It is better suited to our purposes since we will mostly be dealing here with symplectic micromorphisms $([V],\phi):\Cot A\rightarrow \Cot B$ between cotangent microbundles. In this case, we have a canonical lagrangian subbundle
\begin{eqnarray}
	\Lambda & := & (TZ_{A}\oplus 0)\times (0\oplus V(\Cot B)),\label{eq:lagr-distrib}
\end{eqnarray}
where  $V(\Cot B)$  is the vertical subbundle of $\Cot B$. The symplectic micromorphisms from $\Cot A$ to $\Cot B$ are thus the lagrangian submicrofolds $[V,\graph\phi]$ that are transverse to $\Lambda$. 
\end{rem}

The symplectic microbundles and symplectic micromorphisms between them form a symmetric monoidal category; see \cite{SM_Microfolds}. Let us recall here that the tensor product of two microfolds is derived from the usual  cartesian product of manifolds in the obvious way:
\[
	[M,A]\otimes [N,B] := [M\times N, A\times B].
\]
In turn, the tensor product of two symplectic micromorphisms is given by the tensor product of their underlying submicrofolds as above.

\subsection{Equivalence theorem}

We now define two \textit{a priori} different classes of lagrangian submicrofolds $[L,C]\subset\Cot A$, each of which has the conormal microbundle $N^{*}C$ as one of its members: the class of \textit{strongly transverse lagrangian submicrofolds} and the class of \textit{conormal microbundle deformations}. Theorem \ref{thm: equivalence theorem} proves that these classes coincide with one another and with the class of clean lagrangian submicrofolds. 

\begin{defn}
\label{def: W-distribution}Let $C$ be a submanifold of $A$ and let $W\rightarrow C$ be a complementary subbundle to $TC$ in $TA_{|C}$. We define the subbundle $\Lambda^{W}\rightarrow C$ of $T\Cot A_{|C}$ by setting
\begin{equation}
	\Lambda^{W}:=W^{0}\oplus W,\label{eq: W-distribution}
\end{equation}
where $W_{c}^{0}$ is the annihilator of $W_{c}$ (i.e. the subspace of the covectors in $\Cot_{c}A$ that vanish on $W_{c}$) for each $c\in C$. 
\end{defn}

\begin{lem}
\label{lem: W-distributions}All lagrangian subbundles of $T\Cot A_{|C}$ of the form \eqref{eq: W-distribution} are transverse to $N^{*}C$ along $C$.
\end{lem}

\begin{proof}
One checks that the symplectic orthogonal $(\Lambda^{W})_{(0,c)}^{\perp}$ is contained in $\Lambda_{(0,c)}^{W}$ , and we conclude $\Lambda_{(0,c)}^{W}$ is a lagrangian subspace by dimension count. To see that $\Lambda^{W}$ is transverse to $N^{*}C$, we observe that $T_{(0,c)}N^{*}C=N_{c}^{*}C\oplus T_{c}C$. By definition, $W_{c}$ is transverse to $T_{c}C$. This implies that their respective annihilator $W_{C}^{0}$ and $\mbox{\ensuremath{N_{c}^{*}C}}$ are also transverse in $T_{c}^{*}A$. 
\end{proof}

\begin{defn}
We say that a lagrangian submicrofold $[L,C]$ of $\Cot A$ is \textbf{strongly transverse} if it is transverse to all lagrangian subbundles of the form \eqref{eq: W-distribution}. 
\end{defn}

\begin{example}
The conormal bundle $N^{*}C$ is strongly transverse. 
\end{example}

Let $\theta$ be a one-form on $A$, which we regard as a map $\theta:A\rightarrow\Cot A$. Its image $\im\,\theta$ is a lagrangian submanifold of $\Cot A$ if and only if $\theta$ is closed. We say that $\im\,\theta$ is a \textbf{projectable} lagrangian submanifold of $\Cot A$ since the restriction of the canonical projection $\Cot A\rightarrow A$ to $\im\,\theta$ is a diffeomorphism. Conversely, all projectable lagrangian submanifolds of $\Cot A$ are of this form. 

Now if $\theta$ vanishes on some submanifold $C\subset A$, then $\im\,\theta$ contains $C$ (or, more precisely, contains the corresponding submanifold of the zero section $Z_{A}$), and we can consider the induced lagrangian submicrofold $[\im\:\theta,C]$ , which depends only on the germ of $\theta$ around $C$.

The class of lagrangian submicrofolds $[L,C]$ in $\Cot A$ whose image through some special type of symplectomorphism germ from $[\Cot A,N^{*}C]$ to $\Cot N^{*}C$ is projectable will be very important for us in the sequel:

\begin{defn} \label{def: conormal deformation}
We say that a lagrangian submicrofold $[L,C]$ is a\textbf{ deformation} of $N^{*}C$ (or a conormal microbundle deformation) if, for all symplectomorphism germs
\begin{equation}
	[K]:\Cot N^{*}C\longrightarrow[\Cot A,N^{*}C] \label{eq: lag. embedding germ}
\end{equation}
fixing the core and such that the image by $TK$ of the vertical distribution along $C$ in $\Cot N^{*}C$ is of the form \eqref{eq: W-distribution}, there is a germ $[\beta]$ around $C$ of a closed one-form $\beta\in\Omega^{1}(N^{*}C)$ vanishing on $C$ and such that
\begin{equation}
	L=(K\circ T_{\beta}\circ K^{-1})(N^{*}C),\label{eq: sympl. rotation}
\end{equation}
where $T_{\beta}$ is the symplectomorphism germ on $[\Cot N^{*}C,C]$ obtained from $\beta$ by fiber translation as illustrated in Figure \ref{deformation}.
\end{defn}

\begin{rem}
\eqref{eq: sympl. rotation} is equivalent to $L=K(\im\beta)$. 
\end{rem}

\begin{figure}[h!] 
\labellist \small\hair 2pt 
\pinlabel $\Cot A$ at 20 180 
\pinlabel $N^*C$ at 110 230
\pinlabel $V(\Cot N^*C)$ at 470 230  
\pinlabel $[K]$ at 280 180 
\pinlabel $\Cot N^*C$ at 550 180 
\pinlabel $L$ at 187 160 
\pinlabel $T_\beta$ at 400 130 
\pinlabel $A$ at 187 95
\pinlabel $C$ at 123 103
\pinlabel $C$ at 460 100
\pinlabel $N^*C$ at 370 95 
\pinlabel $T_\beta$ at 540 80 
\pinlabel $\Lambda^W$ at 180 10 
\endlabellist 
\centering 
\includegraphics[scale=0.6]{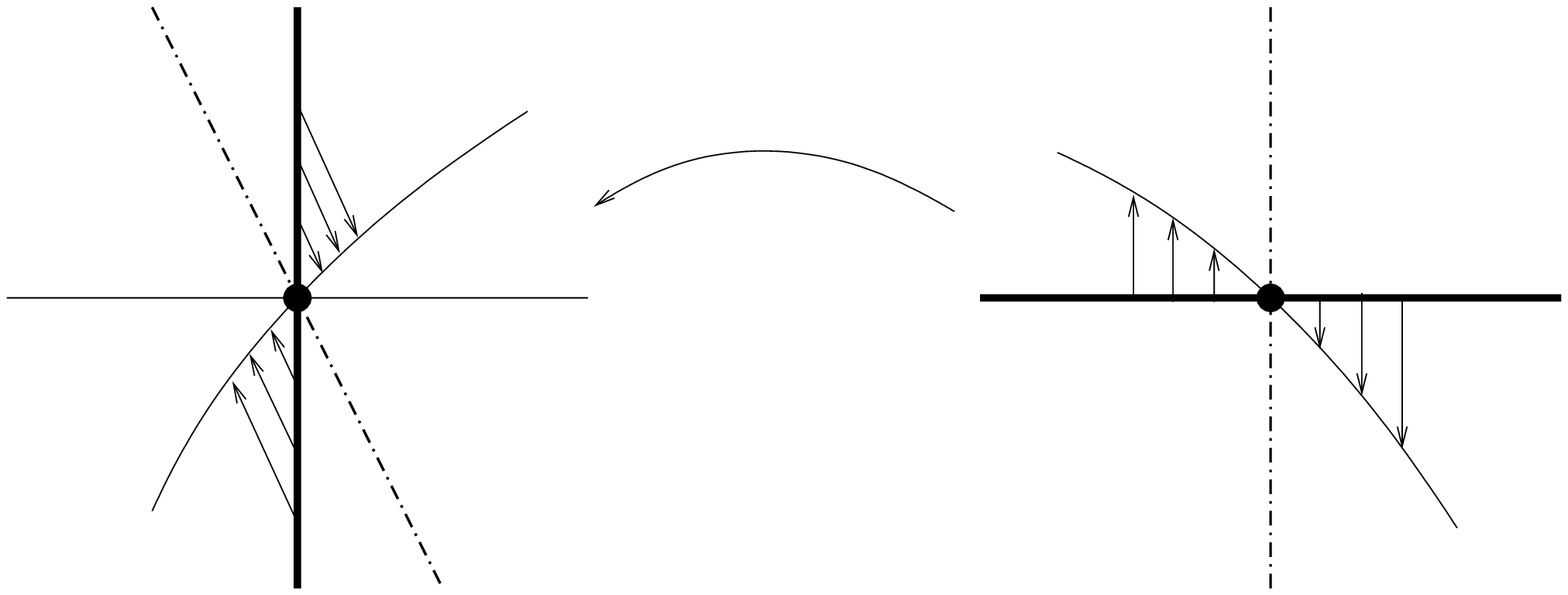} 
\caption{\label{deformation}}
\end{figure}

\begin{thm}\label{thm: equivalence theorem}
Let $[L,C]$ be a lagrangian submicrofold of $[\Cot A,Z_{A}]$. Then the following statements are equivalent:

(1) $[L,C]$ is clean,

(2) $[L,C]$ is strongly transverse,

(3) $[L,C]$ is a deformation of $N^{*}C$.
\end{thm}

\begin{proof}
We first recall that, for  three subspaces, $E,F$ and $G$, of a given vector space, we have 
\begin{equation}
	(E+F)\cap G=E\cap G+F\quad\textrm{iff}\quad F\subset G.\label{eq: pseudo distributivity}
\end{equation}
We start by showing that cleanliness is equivalent to strong transversality.

$(1)\Rightarrow(2)$: Let $\Lambda^{W}\rightarrow C$ be a lagrangian subbundle as in \eqref{eq: W-distribution}. Using $T_{c}A=W_{c}\oplus T_{c}C$ and identity \eqref{eq: pseudo distributivity}, we see that the cleanliness assumption is equivalent to
\begin{equation}
	0\oplus T_{c}C=(0\oplus W_{c})\cap T_{(0,c)}L+0\oplus T_{c}C,\label{eq: equivalent cleaness}
\end{equation}
since $0\oplus T_{c}C$ is contained in $T_{(0,c)}L$. The transversality of $W_{c}$ and $T_{c}C$ implies that the two terms in the R.H.S. of \eqref{eq: equivalent cleaness} intersect only in $\{0\}$. Therefore we can conclude that
\begin{equation}
	(0\oplus W_{c})\cap T_{(0,c)}L=\{0\}.\label{eq: x}
\end{equation}
By taking the symplectic orthogonal of \eqref{eq: x}, we obtain
\begin{eqnarray*}
   T_{(0,c)}\Cot A & = & W_{c}^{0}\oplus T_{c}A+T_{(0,c)}L,\\
                            & = & W_{c}^{0}\oplus T_{c}C+W_{c}^{0}\oplus W_{c}+T_{(0,c)}L,\\
                            & = & W_{c}^{0}\oplus W_{c}+T_{(0,c)}L,
\end{eqnarray*}
where the last equality comes from the identity $W_{c}^{0}\oplus T_{c}C=W_{c}^{0}\oplus0+0\oplus T_{c}C$ whose first term is contained in $W_{c}^{0}\oplus W_{c}$ while its second term is contained in $T_{(0,c)}L$. 

$(2)\Rightarrow(1)$: The strong transversality assumption implies that $L$ is transverse to all lagrangian subbundles $\Lambda^{W}\rightarrow C$ as in \eqref{eq: W-distribution}. Choose one. The symplectic orthogonal of the transversality condition yields 
\begin{equation}
	T_{(0,c)}L\cap(W_{c}^{0}\oplus W_{c})=\{0\},\label{eq:y}
\end{equation}
and, because $0\oplus W_{c}^{0}$ is contained in $W_{c}^{0}\oplus W_{c}$, we have that $T_{(0,c)}L$ intersects $0\oplus W_{c}$ only in $\{0\}$. Now, using this together with \eqref{eq: pseudo distributivity}, we obtain
\begin{eqnarray*}
   T_{(0,c)}L\cap(0\oplus T_{c}A) & = & T_{(0,c)}L\cap(0\oplus W_{c}+0\oplus T_{c}C),\\
                                                    & = & T_{(0,c)}L\cap(0\oplus W_{c})+0\oplus T_{c}C,\\
                                                    & = & 0\oplus T_{c}C,
\end{eqnarray*}
which proves that the intersection $L\cap Z_{A}=C$ is clean. 

The equivalence between $(2)$ and $(3)$ is almost clear from the definitions. Namely, the tangent map of a symplectomorphism germ $K$ as in Definition \ref{def: conormal deformation} maps the vertical bundle in $\Cot N^{*}C$ along $C$ to a lagrangian subbundle $\Lambda^{W}\rightarrow C$ of the form \eqref{eq: W-distribution}. Clearly, a lagrangian submicrofold $[L,C]$ of $\Cot A$ is transverse to $\Lambda^{W}\rightarrow C$ if and only if $[K^{-1}(L),C]$ is transverse to the vertical distribution along $C$. By continuity, this is equivalent to the existence of a (small enough) representative $L$ such that $K^{-1}(L)$ is transverse to the vertical distribution, or, in other words, such that $K^{-1}(L)$ is projectable onto the zero section in $\Cot N^{*}C$. With this in mind, the implication $(2)\Rightarrow(3)$ is clear. The converse follows from the lagrangian embedding theorem which guarantees the existence of a $K$ that sends the vertical distribution to any lagrangian subbundle along $N^{*}C$ and transverse to it in $\Cot A$. 
\end{proof}

\begin{cor}
\label{cor: equivalence}Let $[L,C]\subset\Cot A$ be a lagrangian submicrofold. Then the following statements are equivalent:

(1) $[L,C]$ is clean,

(2') $[L,C]$ is transverse to a lagrangian subbundle $\Lambda^{W}\rightarrow C$ as in  \eqref{eq: W-distribution},

(3') $[L,C]$ is the image of $N^{*}C$ by a symplectomorphism germ
fixing the core as in \eqref{eq: sympl. rotation}.
\end{cor}

\begin{proof}
A closer look at the $(2)\Rightarrow(1)$ part in the proof of Theorem \ref{thm: equivalence theorem} shows that we actually use the weaker version $(2')$ of $(2)$ to prove $(1)$. So we have a cyclic sequence of implications $(2)\Rightarrow(2')\Rightarrow(1)\Rightarrow(2)$ showing that $[L,C]$ is clean iff it is transverse to a lagrangian distribution $\Lambda^{W}$. Similarly, we have that $(3')\Rightarrow(2')\Rightarrow(2)\Rightarrow(3)\Rightarrow(3')$. 
\end{proof}

\begin{rem}
By the relative Poincar\'e Lemma, the one-form germ $[\beta]$ in Definition \ref{def: conormal deformation} is exact, that is, there is a function germ $[S]:N^{*}C\rightarrow\R$ such that $\beta=dS$. From now on, we will remove the ambiguity in the choice of $[S]$ by requiring that it vanishes on $C$. We will call $[S]$ the \textbf{generating function} of the clean lagrangian submicrofold $[L,C]$ \textit{associated with the symplectomorphism germ} $K$. 
\end{rem}

\subsection{Examples}

\subsubsection{Morse-Bott germs}

Consider a smooth function $f:A\rightarrow\R$ . The image $\im\, df$ of its differential $df:A\rightarrow\Cot A$ is a lagrangian submanifold $\Cot A$. The \textbf{critical set }of $f$ is the set
\[
	C_{f}:=(df)^{-1}(Z),
\]
where $Z$ is the zero section in $\Cot A$. We would like to characterize the class of smooth functions on $A$ for which $[\im\, df,C_{f}]$ is a clean lagrangian submicrofold. For this to make sense, we need $C_{f}$ to be a submanifold\footnote{We allow the connected components of $C_f$ to be submanifolds of possibly different dimensions.} of $A$, in which case we call it the \textbf{critical submanifold} of $f$. The cleanliness of $[\im\, df,C_{f}]$ is related to the following notion:

\begin{defn}
We say that the critical submanifold $C_{f}$ is \textbf{nondegenerate} if $\ker\mathcal{H}_{x}f=T_{x}C$ for all $x\in C$, where $\mathcal{H}_{x}f$ is the Hessian of $f$ at $x$ (which we see as a linear map from $T_{x}A$ to $\Cot_{x}A$). A function $f:A\rightarrow\R$ whose critical submanifold is nondegenerate is called a \textbf{Morse-Bott function}.  (If the components of $C_f$ are isolated points, $f$ is a Morse function.)
\end{defn}

\begin{prop}
Let $f:A\rightarrow\R$ be a smooth function whose critical set $C_{f}$ is a submanifold of $A$. Then the lagrangian submicrofold $[\im\, df,C_{f}]$ is clean if and only if $C_{f}$ is nondegenerate. 
\end{prop}

\begin{proof}
Consider the tangent map $Tdf:TA\rightarrow T(\Cot A)$ and let $x\in C_{f}$. On the one hand, we have that
\[
	\im\, T_{x}df=\big\{\big((\mathcal{H}_{x}f)v,v\big):\, v\in T_{x}A\big\},
\]
and, on the other hand, we know that $T_{(0,x)}\im\, df=\im\, T_{x}df$. So the intersection of $T_{(0,x)}\im\, df$ with $0\oplus T_{x}A$ consists of the vectors $0\oplus v$ such that $(\mathcal{H}_{x}f)v=0$, which, thanks to our assumption on $\ker\mathcal{H}_{x}$, is exactly the subspace $0\oplus T_{x}C_{f}$. 
\end{proof}

\begin{example}
Consider a function $f:\R^{n}\rightarrow\R$. We see its differential as the map $df:x\mapsto(\nabla f(x),x)$, where $\nabla f$ is the gradient of $f$. The critical set $C_{f}$ are the points in $\R^{n}$ where the gradient vanishes. The Hessian at $x\in\R^{n}$ is the Jacobian matrix $\mathcal{H}_{x}f=\left(\frac{\partial^{2}f}{\partial x_{i}\partial x_{j}}(x)\right)$ regarded as a linear map from $\R^{n}$ to $\R^{n}$. In case $C_{f}$ is a submanifold, it is nondegenerate when the Jacobian matrix vanishes only on vectors tangent to $C_{f}$. In particular, if $df$ is transverse to the zero section in $\Cot\R^{n}$, $C_{f}$ is a discrete collection of points, and the nondegeneracy condition for $C_{f}$ corresponds to the nondegeneracy of the Jacobian matrix at the critical points. 
\end{example}

\begin{example}
Among the polynomial functions $p_{n}(x)=x^{n}$ on the real line, the only one that yields a clean lagrangian submicrofold $[\im\, dp_{n},0]$ is the quadratic one. 
\end{example}

\begin{defn}
We say that the smooth function germ $[f]:[A,C]\rightarrow[\R,0]$ is a \textbf{Morse-Bott} \textbf{germ} with critical submanifold $C$ if there is a representative $f\in[f]$ having $C$ has its nondegenerate critical submanifold. 
\end{defn}

In other words, the lagrangian submicrofold $[\im\, df,C]$ is clean if and only if $[f]:[A,C]\rightarrow[\R,0]$ is a Morse-Bott germ.

\subsubsection{Transverse lagrangian submicrofolds}\label{sec: trans. lag. subs}

Suppose that the lagrangian submicrofold $[L,C]$ of $\Cot A$ intersects the zero section transversally in $C$, that is,
\[
	T_{(0,c)}L+0\oplus T_{c}A=\Cot_{c}A\oplus T_{c}A,
\]
for all $c\in C$. Of course, this implies that $[L,C]$ is clean, but now we also have that $C$ must be discrete, since two transverse lagrangian submanifolds can intersect only in isolated points. For instance, if $C$ is reduced to a single point $x$, $[L,x]$ is clean iff it is transverse to $T_{(0,x)}Z_{A}$. Therefore the clean lagrangian submicrofolds whose core is a single point correspond precisely to the symplectic micromorphisms from $[\Cot A,A]$ to the cotangent bundle of the one point manifold. Theorem \ref{thm: equivalence theorem} tells us that $[L,x]$ is a deformation of the conormal microbundle of $x$, that is, $[\Cot_{x}A,x]$. In other words, $L$ is the image of $\Cot_{x}A$ by a symplectomorphism germ around $x$ fixing this point and the zero section.

\begin{example}
In the cotangent bundle  $\Cot\R=\R_{p}\oplus\R_{x}$ of the real line, consider the  clean lagrangian submicrofolds $[L,0]$ with the origin as core. The transversality of $[L,0]$ tells us that the projection $(p,x)\mapsto p$ maps a representative $L$ diffeomorphically onto a neighborhood of $0$ in the fiber $\Cot_{0}\R$. Now suppose further that our clean lagrangian submicrofold is the image of $df$ for a Morse-Bott germ $[f]:[\R,0]\rightarrow[\R,0]$. Then the projection $(p,x)\mapsto x$  maps $\im\, df$ diffeomorphically onto a neighborhood of $0$ in $\R$. So the class of Morse-Bott germs $[f]:[\R,0]\rightarrow[\R,0]$ with $df(0)=0$ corresponds to the lagrangian submanifold germs through the origin that are projectable simultaneously on both the $p$-fiber and $x$-fiber.
\end{example}

\section{Local form \label{sub:Local-theory}}

In this section, we show that the lagrangian submicrofold underlying a symplectic micromorphism is clean. As a consequence of the equivalence theorem for clean lagrangian submicrofolds, we find that a symplectic micromorphism is always a deformation of the cotangent lift of its core. This allows us to associate a global generating function (depending on the choice of some symplectomorphism germ) with any symplectic micromorphism. This determines its local form in terms of local generating functions in admissible local charts (Theorem \ref{thm: local form}). Finally, we prove a theorem (Theorem \ref{thm: normal form}) that gives a decomposition of the underlying lagrangian submicrofold of a symplectic micromorphism as a fibration over its core, the fibers of which are actual graphs of smooth map germs.

\subsection{Global generating functions\label{sub:Symplectic-micromorphisms}}

Following \cite{BW1997}, we call \textbf{Schwartz transform} the symplectomorphism
\begin{eqnarray*}
   \mathcal{S}:\overline{\Cot A}\times\Cot B   & \longrightarrow & \Cot(A\times B),\\
   \big((p_{1},x_{1}),\,(p_{2},x_{2})\big)            & \mapsto            & (-p_{1},p_{2},x_{1},x_{2}).
\end{eqnarray*}
The Schwartz transform gives a one-to-one correspondence between the canonical relations from $\Cot A$ to $\Cot B$ and the lagrangian submanifolds of $\Cot(A\times B)$. This remains true in the microworld:

\begin{thm}
\label{pro: sympl. micromorph. as trans. lags.}The Schwartz transform induces a one-to-one correspondence between the symplectic micromorphisms from $\Cot A$ to $\Cot B$ with core $\phi:B\rightarrow A$ and the clean lagrangian submicrofolds in $\Cot(A\times B)$ with core $\graph\phi$.
\end{thm}

\begin{proof}
Thanks to Corollary \ref{cor: equivalence}, we only need to show that the distribution $\Lambda$ in \eqref{eq:lagr-distrib} is of the form $\Lambda^{W}$ as in Definition \ref{def: W-distribution} for some $W$. To begin, we observe that the distribution $W:=\phi^{*}(TA)$ along $\graph\phi$ is transverse to $T\graph\phi$ in $T(A\times B)$. Its annihilator is $W^{0}=\{0\}\oplus V(\Cot B)$, and we verify that $\Lambda^{W}$ coincides with $\Lambda$ when restricted to $\graph \phi$. 
\end{proof}

The proposition above allows us to apply  Theorem \ref{thm: equivalence theorem} to symplectic micromorphisms $([V],\phi):\Cot A\rightarrow\Cot B$. The \textbf{cotangent lift} $\Cot\phi:\Cot A\rightarrow\Cot B$ of the core map plays the role of the conormal microbundle since, by definition, 
\[
	\Cot\phi:=\mathcal{S}^{-1}\big(N^{*}\graph\phi\big).
\]
The conormal bundle $N^{*}\graph\phi\subset\Cot(A\times B)$ is the image of the lagrangian embedding $\iota_{\phi}$ of $\phi^{*}(\Cot A)$ into $\Cot(A\times B)$ given by
\begin{eqnarray}\label{eq:cotangent-lift-inclusion-1}
	\iota_{\phi}(p_{1},x_{2}) & = & \Big(\big(p_{1},-(\Cot_{x_{2}}\phi)p_{1}\big),\big(\phi(x_{2}),x_{2}\big)\Big).
\end{eqnarray}
Thus the cotangent lift can be described by the representative 
\[
	\Cot\phi:=\Big\{\Big(\big(p_{1},\,\phi(x_{2})\big),\,\big((\Cot_{x_{2}}\phi)p_{1},\, x_{2}\big)\Big):\,(p_{1},x_{2})\in\phi^{*}(\Cot A)\Big\}.
\]
This allows us to identify the cotangent lift $\Cot\phi$ and the conormal bundle $N^{*}\graph\phi$ with the pullback bundle $\phi^{*}(\Cot A)$. 

We can now apply the equivalence theorem to the case of symplectic micromorphisms. Namely, Theorem \ref{thm: equivalence theorem} tells us that the data of a symplectomorphism germ
\begin{equation}
	K_{\phi}:\Cot(\phi^{*}(\Cot A))\longrightarrow\Big[\overline{\Cot A}\times\Cot B,\,\Cot\phi\Big],\label{eq: K-phi}
\end{equation}
whose restriction to the zero section coincides with $\iota_{\phi}$ and which satisfies the condition in Definition \ref{def: conormal deformation}, allows us to describe a symplectic micromorphism $([V],\phi)$ from $\Cot A$ to $\Cot B$ in the two following ways: 

\begin{itemize}
\item $[V]=R_{\phi}^{V}(\Cot\phi)$ for a symplectomorphism germ
\[
	R_{\phi}^{V}:\Big[\overline{\Cot A}\times\Cot B,\,\graph\phi\Big]\rightarrow\Big[\overline{\Cot A}\times\Cot B,\,\graph\phi\Big],
\]
which is uniquely determined by $K_{\phi}$ and $([V],\phi)$. In other words, each symplectic micromorphism is a deformation in this sense of the cotangent lift of its core.

\item $[V]=K_{\phi}^{-1}(\im[df])$ for a smooth function germ
\[
	[f]:[\phi^{*}(\Cot A),Z_{B}]\rightarrow[\R,0]
\]
whose critical submanifold is the zero section $Z_{B}$ of $\phi^{*}(\Cot A)$. The germ $[f]$ is uniquely determined by $K_{\phi}$ and $([V],\phi)$, and it is called the \textbf{global generating function} of $([V],\phi)$ associated with $K_{\phi}$. 
\end{itemize}

As before, the two descriptions are related: namely,
\[
	R_{\phi}^{V}=K_{\phi}\circ T_{df}\circ K_{\phi}^{-1},
\]
where $T_{df}$ is again the symplectomorphism germ obtained by fiber translation with the one form germ $df$.

\begin{example} \label{exa: local gen fct 1}
Let $\phi$ be a smooth map from an open subset $U_{2}\subset\R^{l}$ to an open subset $U_{1}\subset\R^{k}$. In this case, we have a canonical symplectomorphism germ
\begin{eqnarray*}
	\Cot(\phi^{*}(\Cot U_{1}))                  & \overset{K_{\phi}}{\rightarrow} & \Big[\overline{\Cot U_{1}}\times\Cot U_{2},\,\Cot\phi\Big]\\
	\big((v_{1},p_{2}),(p_{1},x_{2})\big)  & \mapsto                                    & \big(p_{1},\phi(x_{2})+v_{1},\,(\Cot_{x_{2}}\phi)p_{1}+p_{2},x_{2}\big)
\end{eqnarray*}
since we can identify $\Cot(\phi^{*}(\Cot U_{1}))$ with $\R_{v_{1}}^{k}\times\R_{p_{2}}^{l}\times\R_{p_{1}}^{k}\times U_{2}$. Therefore, a symplectic micromorphism $([V],\phi)$ from $\Cot U_{1}$ to $\Cot U_{2}$ is completely determined by the germ of a function
\begin{equation}
	[f]:[\R_{p_{1}}^{k}\times U_{2},\{0\}\times U_{2}]\rightarrow[\R,0]\quad\textrm{s.t.}\quad\partial_{p}f(0,x_{2})=0.\label{eq: local gen. fcts}
\end{equation}
In very explicit terms, a representative $V$of the symplectic micromorphism can be described as the set of points in $\overline{\Cot U_{1}}\times\Cot U_{2}$ of the form
\begin{equation}
	\Big(p_{1},\phi(x_{2})+\partial_{p}f(p_{1},x_{2}),\,(\Cot_{x_{2}}\phi)p_{1}+\partial_{x}f(p_{1},x_{2}),x_{2}\Big),\label{eq: local form}
\end{equation}
where $(p_{1},x_{2})$ runs in a suitable neighborhood of the zero section in $\phi^{*}(\Cot U_{1})$. Here, we see that the symplectomorphism germ
\[
	R_{V}^{\phi}:\Big[\overline{\Cot U_{1}}\times\Cot U_{2},\,\graph\phi\Big]\rightarrow\Big[\overline{\Cot U_{1}}\times\Cot U_{2},\,\graph\phi\Big]
\]
given by the formula
\begin{equation}
	R_{\phi}^{V}(p_{1},x_{1},p_{2},x_{2})=\big(p_{1,}x_{1}+\partial_{p}f(p_{1},x_{2}),\, p_{2}+\partial_{x}f(p_{1},x_{2}),x_{2}\big)\label{eq: R local}
\end{equation}
maps $\Cot\phi$ diffeomorphically onto $([V],\phi)$ as prescribed by the equivalence theorem. 
\end{example}

\subsection{Local generating functions}

Example \ref{exa: local gen fct 1} sets us on the way toward a notion of local generating function for symplectic micromorphisms.

\begin{defn}
An \textbf{admissible local chart} for a symplectic micromorphism $([V],\phi):\Cot A\rightarrow\Cot B$ is a local chart of $\Cot A\times\Cot B$ with domain of the form $\Cot U$, where both factors of $U:=U_{1}\times U_{2}$ are the domains of coordinate patches $\chi_{1}:U_{1}\rightarrow A$ and $\chi_{2}:U_{2}\rightarrow B$ such that $\phi(\chi_{2}(U_{2}))\subset\chi_{1}(U_{1})$. 

We define the \textbf{restriction} of $([V],\phi)$ to $\Cot U$ to be the symplectic micromorphism $([V_{U}],\phi_{U})$ from $\Cot U_{1}$ to $\Cot U_{2}$ obtained as the image
\[
	([V_{U}],\phi_{U}):=(\Cot\chi_{1}\times\Cot\chi_{2})\Big(\Big[V\cap\Cot(A\times B)_{|\chi(U)},\graph\phi\cap\chi(U)\Big]\Big),
\]
where $\chi:=\chi_{1}\times\chi_{2}$. 
\end{defn}

Since the restriction of a symplectic micromorphism to an admissible local chart is a symplectic micromorphism and admits a generating function as in Example \ref{exa: local gen fct 1}, we immediately obtain the following ``local form'' theorem:

\begin{thm}
\label{thm: local form}(Local form). Let $([V],\phi):\Cot A\rightarrow\Cot B$ be symplectic micromorphism, and let $\Cot U$ be an admissible local
chart. Then there is a representative $V_{U}\in[V_{U}]$ of the restriction to the local chart such that 
\begin{eqnarray}
	V_{U} & = & \Big\{\Big(p_{1},\partial_{p}F(p_{1},x_{2}),\partial_{x}F(p_{1},x_{2}),x_{2}\Big):(p_{1},x_{2})\subset W\Big\},\label{eq: F local form}
\end{eqnarray}
where $W$ is a suitable neighborhood of the zero section in $\phi^{*}(\Cot U_{1})$. The \textbf{local generating function} $F$ is of the form
\begin{eqnarray}
	F(p_{1},x_{2}) & := & \langle p_{1},\phi(x_{2})\rangle+f(p_{1},x_{2}),\label{eq:Ff}
\end{eqnarray}
where $f$ is a representative of the function germ as in \eqref{eq: local gen. fcts}.
\end{thm}

\begin{example}
Let $\Cot\phi:\Cot A\rightarrow\Cot B$ be the cotangent lift of a smooth map $\phi:B\rightarrow A$. In an admissible local chart $\Cot U$, the cotangent lift admits the local generating function
\[
	F_{U}(p_{1},x_{2}):=\langle p_{1},\phi_{U}(x_{2})\rangle.
\]
However, the global generating function $[f]:\phi^{*}(\Cot A)\rightarrow[\R,0]$ associated by the equivalence theorem with any symplectomorphism germ \eqref{eq: K-phi} is always zero for cotangent lifts. 
\end{example}

\subsection{Composition formula and monicity}

Let $M,N$ and $P$ be three sets and let $V\subset M\times N$ and $W\subset N\times P$ be two binary relations. We say that the composition $W\circ V\subset M\times P$ is \textbf{monic} if, for all $z=(m,p)\in W\circ V$, the set
\[
	\mathcal{M}_{z}=\big\{ n\in N:\,(m,n)\in V,\,(n,p)\in W\big\}
\]
is a singleton. In the microsymplectic world, the corresponding definition is the following:

\begin{defn}
Let $[M,A]$, $[N,B]$ and $[P,C]$ be three symplectic microfolds. The composition of the symplectic micromorphisms
\[
	[M,A]\overset{([V],\phi)}{\longrightarrow}[N,B]\overset{([W],\psi)}{\longrightarrow}[P,C]
\]
is \textbf{monic} if there are representatives $V\in[V]$ and $W\in[W]$ whose composition, as binary relations, is monic. 
\end{defn}

Our goal here is to show that the composition of symplectic micromorphisms is \textit{always} monic. By the lagrangian embedding theorem, it is enough to see this for symplectic micromorphisms between cotangent microbundles: 
\[
	\Cot A\overset{([V],\phi)}{\longrightarrow}\Cot B\overset{([W],\psi)}{\longrightarrow}\Cot C.
\]
First of all, for all $c\in C$ and for all representatives $V\in[V]$ and $W\in[W]$, we have that
\[
	z_{c}:=\big((0,(\phi\circ\psi)(c)),(0,c)\big)\in W\circ V,
\]
and $\mathcal{M}_{z_{c}}=\{\psi(c)\}$. We need to check that the composition remains monic in a neighborhood of $z_{c}$. To do this, we can go to local coordinates and consider admissible local charts
\begin{eqnarray*}
	\Cot U_{1}\times\Cot U_{2} & \textrm{of} & \Cot A\times\Cot B,\\
	\Cot U_{2}\times\Cot U_{3} & \textrm{of} & \Cot B\times\Cot C,
\end{eqnarray*}
such that $(\phi\circ\psi)(c)\in U_{1}$, $\psi(c)\in U_{2}$ and $c\in U_{3}$. By Theorem \ref{thm: local form}, we can express the restrictions of these symplectic micromorphisms in terms of local generating functions:
\begin{eqnarray*}
	V_{U}  & = & \Big\{\Big((p_{1},\partial_{p}F(p_{1},x_{2})),(\partial_{x}F(p_{1},x_{2}),x_{2})\Big):\, (p_{1},x_{2})\in N_V\Big\}\\
	W_{U} & = & \Big\{\Big((p_{2},\partial_{p}G(p_{2},x_{3})),(\partial_{x}G(p_{2},x_{3}),x_{3})\Big):\, (p_{2},x_{3})\in N_W\Big\},
\end{eqnarray*}
where $N_W$ is a neighborhood of $(0,x_3)$ in $\psi^*(\Cot U_2)$ and $N_V$ is a neighborhood of $(0,\psi(x_3))$ in $\phi^*(\Cot U_1)$. A point $z\in W_{U}\circ V_{U}$ is of the form
\[
	z=\Big((p_{1},\partial_{p}F(p_{1},\bar{x}_{2})),(\partial_{x}G(\bar{p}_{2},x_{3}),x_{3})\Big),
\]
for $(\bar{p}_{2},\bar{x}_{2})$ such that
\[
	(\partial_{x}F(p_{1},\bar{x}_{2}),\bar{x}_{2})=(\bar{p}_{2},\partial_{p}G(\bar{p}_{2},x_{3})).
\]
The following lemma shows that $\mathcal{M}_{z}$ is reduced to a single point.

\begin{lem}
For all $(p_{1},x_{3})$ with $p_{1}$ small enough, the following system
\begin{eqnarray*}
	\bar{p}_{2} & = & \partial_{x}F(p_{1},\bar{x}_{2}),\\
	\bar{x}_{2} & = & \partial_{p}G(\bar{p}_{2},x_{3}),
\end{eqnarray*}
has a unique solution $\big(\bar{p}_{2}(p_{1},x_{3}),\bar{x}_{2}(p_{1},x_{3})\big)$.
\end{lem}

\begin{proof}
This follows from a straightforward application of the implicit function Theorem to the function
\[
K(p_{1,}p_{2},x_{2},x_{3})=\left(\begin{array}{c}
										\partial_{x}F(p_{1},x_{2})-p_{2}\\
										\partial_{p}G(p_{2},x_{3})-x_{2}
									\end{array}
								\right)
\]
around the point $(0,0,\psi(x_{3}),x_{3})$ since $\partial_{x}F(0,x_{2})=0$ and $\partial_{p}G(0,x_{3})=\psi(x_{3})$. 
\end{proof}

Putting everything together, we obtain the following proposition:

\begin{prop} \label{pro: simplicity}
The composition of symplectic micromorphisms is always monic.
\end{prop}

As a byproduct, we also get a composition formula for local generating functions. We start by noticing that the unique point $(\bar{p}_{2},\bar{x}_{2})$ in $\mathcal{M}_{z}$ is also the unique critical point of the function
\[
	H_{p_{1},x_{3}}(p_{2},x_{2}):=F(p_{1},x_{2})+G(p_{2},x_{3})-p_{2}x_{2},
\]
where $p_{1}$ and $x_{3}$ are held fixed. If we denote by $(\mbox{G\ensuremath{\star}F})(p_{1},x_{3})$ the function $ $$H_{p_{1},x_{2}}$ evaluated at its critical point, we obtain the following result:

\begin{prop}
In the notation as above, the function $G\star F$ is the local generating function of the composition $W_{U}\circ V_{U}$.
\end{prop}

\begin{proof}
We see this by noticing that
\begin{eqnarray*}
	\partial_{p}(G\star F)(p_{1},x_{3}) & = & \partial_{p}F(p_{1},\bar{x}_{2}),\\
	\partial_{x}(G\star F)(p_{1},x_{3}) & = & \partial_{x}G(\bar{p}_{2},x_{3}),
\end{eqnarray*}
where $(\bar{p}_{2},\bar{x}_{2})$ is the unique critical point of $H_{p_{1},x_{3}}$. 
\end{proof}

\subsection{Decomposition in terms of graphs of maps}

A cotangent lift $\Cot\phi:\Cot A\rightarrow\Cot B$ has the nice property that, even though it is not itself the graph of a map when $\phi$ is not a diffeomorphism, the intersection of its underlying lagrangian submicrofold with the coisotropic submanifold $\Cot A\times\Cot_{x_{2}}B$ is the graph of a map: 
\begin{eqnarray*}
	\graph\Cot_{x_{2}}\phi & = & \Cot\phi\cap(\Cot A\times\Cot_{x_{2}}B),\\
	\Cot\phi                        & = & \bigcup_{x_{2}\in B}\graph\Cot_{x_{2}}\phi.
\end{eqnarray*}
To make sense of this map-like property for general symplectic micromorphisms, we need the following definitions:

\begin{defn}
A \textbf{lagrangian fibration} $\mathcal{F}A$ of a symplectic microfold $[M,A]$ is a collection, smoothly parametrized by $A$, of transverse lagrangian submicrofolds (in the sense of Paragraph \ref{sec: trans. lag. subs})
\[
	[\mathcal{F}_{x}A,\{x\}]\subset[M,A],\quad x\in A.
\]
A lagrangian fibration of $[M,A]$ \textbf{along a smooth map} $\phi:B\rightarrow A$ is a collection $\mathcal{F}A$ smoothly indexed by $B$ of transverse lagrangian submicrofolds
\[
	[\mathcal{F}_{y}A,\{\phi(y)\}]\subset[M,A],\quad y\in B.
\]
The \textit{vertical} lagrangian fibration in a cotangent microbundle is given by the germ of its fibers at $0$. 
\end{defn}

\begin{example}
Consider the symplectic micromorphism $T=([\graph\Psi],\Psi_{|A}^{-1})$ from $[M,A]$ to $[N,B]$ coming from a symplectomorphism germ $\Psi:[M,A]\rightarrow[N,B]$ between two symplectic microfolds. For any lagrangian fibration $\mathcal{F}B$ of $[N,B]$, we obtain a corresponding lagrangian fibration $\mathcal{F}A$ of $[M,A]$ along $\phi:=\Psi_{|A}^{-1}$ by setting
\[
	\mathcal{F}_{y}A:=\Psi^{-1}(\mathcal{F}_{y}B).
\]
If we denote by $\Psi_{y}$ the restriction of $\Psi$ to $\mathcal{F}_{y}A$, we obtain the identities 
\begin{eqnarray*}
	\graph\Psi_{y} & = & \graph\Psi\cap(M\times\mathcal{F}_{y}B),\\
	\graph\Psi       & = & \bigcup_{y\in B}\graph\Psi_{y},
\end{eqnarray*}
which are similar to the ones we had in the case of cotangent lifts, except that here the lagrangian fibration $\mathcal{F}B$, (which now plays the role of the vertical distribution for cotangent microbundles) is not canonical. 
\end{example}

\begin{thm}
\label{thm: normal form} Let $([V],\phi):[M,A]\rightarrow[N,B]$ be a symplectic micromorphism between two symplectic microfolds. Then, for any lagrangian fibration $\mathcal{F}B$ of $[N,B]$, there exists a unique corresponding lagrangian fibration $\mathcal{F}A$ along $\phi$ such that
\begin{eqnarray*}
   \graph\Psi_{y} & = & V\cap(P\times\mathcal{F}_{y}B),\\
   V                     & = & \bigcup_{y\in B}\graph\Psi_{y},
\end{eqnarray*}
for suitable representatives, and where $[\Psi_{y}]:[\mathcal{F}_{y}A,\{\phi(y)\}]\rightarrow[\mathcal{F}_{y}B,\{y\}]$ is a collection of smooth map germs indexed by $B$. 
\end{thm}

\begin{rem} \label{rem: composition}
The lagrangian fibration $\mathcal{F}B$ of $[N,B]$ in the theorem above gives rise to the collection of symplectic micromorphisms
\[
	\big(\big[\mathcal{F}_yB\big], c_y\big): [N,B]\longrightarrow E,\quad y\in B,
\]
where $E=\{0\}\times \{*\}$ is the cotangent microbundle of the one-point manifold $\{*\}$. The core maps are the constant functions $c_y:\{*\}\rightarrow B$ that map the unique point of the core of $E$ to each $y$. The corresponding lagrangian fibration $\mathcal{F}A$ along $\phi$ is then obtained from the symplectic micromorphism $([V],\phi)$ by composition
\begin{equation}
	\big(\big[\mathcal{F}_yA\big], c_{\phi(y)}\big) = \big(\big[\mathcal{F}_yB\big], c_y\big)\circ ([V],\phi ).\label{eq: lag fib as mic}
\end{equation}
\end{rem}

\begin{proof}
The uniqueness of the decomposition is immediate from Remark \ref{rem: composition}. As for the existence, let $\mathcal{F}B$ be a lagrangian fibration of $[N,B]$ and consider the lagrangian fibration $\mathcal{F}A$ along $\phi$ as defined by \eqref{eq: lag fib as mic}. We denote by $R_{y}$ the intersection of $V$ with $P\times\mathcal{F}_{y}B$. By Remark \ref{rem: composition}, we have that $R_{y}\subset\mathcal{F}_{y}A\times\mathcal{F}_{y}B$ and that, for appropriate representatives, 
\[
	V=\bigcup_{y\in B}R_{y}.
\]
We need to show that $R_{y}$ is the graph of a map. For this, we consider the composition of the symplectic micromorphisms
\[
	[M,A]\overset{([V],\phi)}{\longrightarrow}[N,B]\overset{\mathcal{F}_{y}B}{\longrightarrow}E.
\]
A point $z\in\mathcal{F}_{y}B\circ V$ is of the form $z=(p,(0,\star))$, where $p\in\mathcal{F}_{y}A$. Since, by Proposition \ref{pro: simplicity}, the composition of two symplectic micromorphisms is alway monic the set $\mathcal{M}_{z}$ is a singleton whose unique point is in $\mathcal{F}_{y}B$. We denote this point by $\Psi_{y}(p)$, and this gives us a map $\Psi_{y}:\mathcal{F}_{y}A\rightarrow\mathcal{F}_{y}B$ whose graph is, by definition, $R_{y}$. 
\end{proof}

\begin{example}
\label{exa: local normal form}Let $([V],\phi):\Cot U_{1}\rightarrow\Cot U_{2}$ be a symplectic micromorphism between cotangent microbundles over the open subsets $U_{1}\subset\R^{k}$ and $U_{2}\subset\R^{l}$. The local form theorem tells us that there is a representative $V\in[V]$ that can be described by a generating function $[f]$ as in \eqref{eq: local form}. Thus $V\cap(\Cot U_{1}\times\Cot_{x_{2}}U_{2})$ is the locus of points of the form
\[
	\Big(p_{1},\phi(x_{2})+\partial_{p}f(p_{1},x_{2}),\,(\Cot_{x_{2}}\phi)p_{1}+\partial_{x}f(p_{1},x_{2}),x_{2}\Big),
\]
where $x_{2}$ is fixed while $p_{1}$ is free to vary in a neighborhood of $0$ in $\R^{k}$. Now, for each $x_{2}\in U_{2}$, we may define the symplectomorphism germ
\begin{eqnarray*}
   [\Cot U_{1},\phi(x_{2})] & \overset{R_{x_{2}}}{\longrightarrow} & [\Cot U_{1},\phi(x_{2})],\\
   (p_{1},x_{1})                 & \mapsto                                             & x_{1}+\partial_{p}f(p_{1},x_{2}).
\end{eqnarray*}
The image of the cotangent bundle fiber over $\phi(x_{2})$ by $R_{x_{2}}$ defines a fiber of our lagrangian fibration of $\Cot U_{1}$ along the core map.  Explicitly, it is
\[
	L_{x_{2}}=\Big\{\big(p_{1},\phi(x_{2})+\partial_{p}f(p_{1},x_{2})\big):p_{1}\in W\Big\},
\]
where $W$ is a suitable neighborhood of $0$ in $\Cot_{\phi(x_{2})}U_{1}$. $L_{x_2}$ is a lagrangian submanifold of $\Cot U_{1}$ that intersects the zero section transversally in $\phi(x_{2})$. Now the map $\Psi_{x_{2}}$ from $L_{x_{2}}$ to $\Cot_{x_{2}}N$ is given by the formula 
\[
	\Psi_{x_{2}}\big(p_{1},\phi(x_{2})+\partial_{p}f(p_{1},x_{2})\big)=\big((\Cot_{x_{2}}\phi)p_{1}+\partial_{x}f(p_{1},x_{2}),x_{2}\big).
\]
\end{example}

\begin{example}
Consider the symplectic micromorphism $([V],\phi)$ from $\R^{2}=\R_{p}\oplus\R_{x}$ to itself whose core is the constant function $\phi(x)=0$ and whose generating function is given by $f(p,x)=p^{2}x$. The map 
\[
	(p,x)\mapsto\big((p,2xp),(p^{2},x)\big)
\]
parametrizes a representative $V\in[V]$. Consider the straight lines
\[
	l_{x}:=\big\{(p,2xp):p\in\R\big\}
\]
in $\R^{2}$ through the origin. The intersection of $V$ with $\R^{2}\times(\R_{p}\oplus\{0\})$ is the graph of the map $\Psi_{x}(p,2xp)=(p^{2},x)$ that folds $l_{x}$ at the origin into a half-line and maps this ray linearly into the half-line parallel to the $p$-axis and passing through $(0,x)$.
\end{example}

\begin{example}
Let $R_{A}:\Cot A\rightarrow\Cot A$ be a symplectomorphism germ that fixes the core; i.e., the core of the corresponding symplectic micromorphism is the identity map on $A$. For each $x\in A$, $R_{A}$ defines two lagrangian distributions, 
\[
	\mathcal{F}_{x}A:=R_{A}(\Cot_{x}A)\quad\textrm{and}\quad\mathcal{B}_{x}A:=R_{A}^{-1}(\Cot_{x}A),
\]
the forward and backward images of the cotangent fibers via $R_{A}$. Clearly, the restriction of $R_{A}$ to the backward distribution yields the decomposition $\Psi_{x}:\mathcal{B}_{x}A\rightarrow\Cot_{x}A$ associated with the cotangent fiber distribution. Now, consider a symplectic micromorphism of the form
\[
	([V],\phi)=\graph[R_{A}^{-1}]\circ\Cot\phi\circ\graph[R_{B}].
\]
where $R_{A}$ and $R_{B}$ are symplectomorphism germs on respectively $\Cot A$ and $R_{B}$ of $\Cot B$ fixing the cores, and where $\Cot\phi:\Cot A\rightarrow\Cot B$ is a cotangent lift. Then we obtain a decomposition given by the following diagram:

\begin{diagram} 
   T^*_{\phi(x_2)}A              & \rTo^{T^*_{x_2}\phi}   & T^*_{x_2} B \\
   \uTo^{R_A^{-1}}               &                                   & \dTo_{R_B}  \\
   \mathcal{B}_{\phi(x_2)}A & \rTo_{\Psi_{x_2}}        & \mathcal{F}_{x_2}B
\end{diagram}

\end{example}

\section{Hamiltonian flows\label{sec:Applications-to-classical}}

In this section, we explain how symplectic microgeometry is a natural framework for the Hamilton-Jacobi theory of hamiltonian flows through their local generating functions. We show that there is a canonical symplectic micromorphism,
\[
	\rho_{H}:\Cot\R\otimes\Cot Q\rightarrow\Cot Q,
\]
the evolution micromorphism, that encodes the short-time dynamics of an hamiltonian system $H:\Cot Q\rightarrow\R$. The local generating function of $\rho_{H}$ in an admissible local chart coincides with the solution of the Hamilton-Jacobi equation for the generating function of the hamiltonian flow in this chart. 

From a different perspective, the evolution micromorphism allows us to formulate the short-time evolution in classical mechanics in a purely categorical way. Namely, we show that $\rho_{H}$ turns $\Cot Q$ into a module over $\Cot\R$ and that all $\Cot\R$-modules arise from hamiltonian flows (with possibly \textit{time}-\textit{dependent} hamiltonians).

\subsection{The evolution micromorphism}

Consider a hamiltonian system $H:\Cot Q\rightarrow\R$. The time evolution $\Psi_{t}:\Cot Q\rightarrow\Cot Q$ generated by $H$ is the flow of the Hamiltonian vector field $X_{H}$. It produces a lagrangian submanifold $W_{H}$ of $\overline{\Cot\R}\times\overline{\Cot Q}\times\Cot Q$, which we call the \textbf{evolution submanifold}, and which is defined as 
\begin{equation}
	W_{H}:=\bigg\{\Big(\big(t,H(\Psi_{t}(z))\big),z,\Psi_{t}(z)\Big):\; t\in I_z,\, z\in\Cot Q\bigg\},\label{eq: evolution submanifold}
\end{equation}
where $I_z$ is the maximal interval on which $\Psi_{t}(z)$ is defined. 

\begin{rem}
The core $\R$ of the $\Cot \R$-factor in the product $\overline{\Cot\R}\times\overline{\Cot Q}\times\Cot Q$ above corresponds to the possible energy levels $E\in\R$ of the system and not to the possible times. To remind us of this fact, we will denote this core by $\Energy$ instead of $\R$. A point $t\in\Cot_{E}\Energy$ in the fiber represents the time, in accordance with the physical time-energy duality. 
\end{rem}

Now we introduce the following map
\begin{eqnarray}
   J:Q & \rightarrow & \Energy\times Q,\label{eq: U-map}\\
   q & \mapsto & (H(0,q),q).\nonumber 
\end{eqnarray}

\begin{rem}
Observe that for ``mechanical'' hamiltonians $H(p,q)=\frac{1}{2}g(q)(p,p)+V(q)$ coming from a metric $g$ and a potential $V$ on $Q$, the map $J$ is essentially the same as the potential. 
\end{rem}

We see that the points in $W_{H}$ with $t=0$ and $z=(0,q)$ lie in the graph of $J$, and therefore it makes sense to consider the lagrangian submicrofold $[W_{H},\graph J]$. A straightforward check shows that, if restricted to sufficiently small times and momenta, the evolution submanifold intersects the zero section cleanly in $\graph J$. This yields the following result:

\begin{prop}
Let $H:\Cot Q\rightarrow\R$ be an hamiltonian system. Then the germ of the evolution submanifold \eqref{eq: evolution submanifold} around the graph of \eqref{eq: U-map} yields a symplectic micromorphism
\begin{eqnarray}
	\big([W_{H}],J\big):\Cot\Energy\otimes\Cot Q & \longrightarrow & \Cot Q,\label{eq: evolution micromorphism}
\end{eqnarray}
which we will refer to as the \textbf{evolution micromorphism} of the hamiltonian system. 
\end{prop}

\begin{rem}
In the discussion above, we started with a global hamiltonian $H:\Cot Q\rightarrow \R$ and we obtained from it its evolution micromorphism $([W_H],J)$, which encodes the short-time dynamics of the hamiltonian system. By doing so, we lost some information since two different global hamiltonians will yield the same evolution submicrofold if their germs around the zero section in $\Cot Q$ coincide. Of course, it would have sufficed to start with a hamiltonian function germ $[H]:[\Cot Q,Z_Q]\rightarrow \R$ and then consider the induced germ of flow to define the evolution submicrofold.
\end{rem}

\subsection{The Hamilton-Jacobi equation}

\subsubsection{Time-independent hamiltonians}

We can always choose the admissible charts for the evolution micromorphism \eqref{eq: evolution micromorphism} to be of the form
\begin{equation}
	\big((t,E),(p,q),(P,Q)\big)\in\Cot\Energy\times\Cot U\times\Cot U,\label{eq: admin local chart}
\end{equation}
where $U$ is a coordinate patch of $Q$. The local form theorem tells us that there exist a unique local generating function $S(t,p,Q)$ of the form
\begin{equation}
	S(t,p,Q)=pQ+tH(0,Q)+f(t,p,Q),\label{eq: S=pQ +tH +f}
\end{equation}
which is defined for sufficiently small times and momenta, and such that 
\begin{equation}
W_{H}^{U}=\Big\{\big((t,\partial_{t}S(t,p,Q)),(p,\partial_{p}S(t,p,Q)),(\partial_{q}S(t,p,Q),Q)\big)\Big\},\label{eq: evolution mic. local form}
\end{equation}
where the variables $(t,p,Q)$ in \eqref{eq: evolution mic. local form} range in a suitable neighborhood of the zero section in $J^{*}(\Cot(\Energy\times Q))$, is a representative of the restriction of evolution micromorphism to the local chart. Comparing \eqref{eq: evolution mic. local form} with \eqref{eq: evolution submanifold}, we see that $S$ satisfies the following partial differential equation
\begin{equation}
	\partial_{t}S(t,p,Q)=H(\partial_{q}S(t,p,Q),Q)\label{eq: HJ equation}
\end{equation}
with initial condition $S(0,p,Q)=\langle p,Q\rangle$. Equation \eqref{eq: HJ equation} is known as the \textbf{Hamilton-Jacobi equation} (for generating functions of type $pQ$) of the hamiltonian system $(\Cot Q,H)$. Here, the existence of a solution is an immediate consequence of the local form theorem applied to the evolution micromorphism.

Now suppose that we are given a generating function $S(t,p,Q)$ that is nondegenerate in the following sense:
\begin{equation}
	\det\big|\frac{\partial^{2}S(0,0,Q)}{\partial p\,\partial Q}\big|\neq0.\label{eq: non-deg for S}
\end{equation}
In this case, the implicit function theorem applied to the function $K(t,p,Q)=\partial_{p}S(t,p,Q)-q$ around the point $(0,0,q)$ guarantees that the following implicit system 
\begin{eqnarray}
	q & = & \partial_{p}S(t,p,Q)\label{eq: x_t}\\
	P & = & \partial_{q}S(t,p,Q)\label{eq: p_t}
\end{eqnarray}
has a unique solution $(P(t),Q(t))$ for each $(t,p,q)$ with $t$ and $p$ small enough. This generates a flow $\Psi_{t}(p,x)=(P(t),Q(t))$ on a neighborhood of the zero section of $\Cot U$. The Hamilton-Jacobi theorem tells us that, if $S$ further satisfies \eqref{eq: HJ equation}, then this flow coincides with the flow generated on the local chart $\Cot U$ by the hamiltonian flow of $H$. 

For the local generating function \eqref{eq: S=pQ +tH +f} of the evolution micromorphism \eqref{eq: evolution micromorphism}, the nondegeneracy \eqref{eq: non-deg for S} is clear from \eqref{eq: S=pQ +tH +f}, and the fact that the implicit Equations \eqref{eq: x_t} and \eqref{eq: p_t} generate the flow is immediate from a comparison of \eqref{eq: evolution mic. local form} with \eqref{eq: evolution submanifold}. 

To conclude this paragraph, we will exhibit some nice relations between the hamiltonian and the local generating functions of its evolution micromorphism.  From the Hamilton-Jacobi equation together with \eqref{eq: S=pQ +tH +f}, we obtain
\begin{eqnarray*}
	H(p,q)    & = & \partial_{t}S(0,p,q),\\
	S(t,p,Q) & = & pQ+\int_{0}^{t}H(P(s),Q(s))ds.
\end{eqnarray*}
Moreover, a straightforward Taylor expansion in the time variable yields
\begin{eqnarray*}
	S(t,p,Q) & = & pQ+H(p,Q)t+\partial_{x}H(p,Q)\partial_{p}H(p,Q)\frac{t^{2}}{2}+\cdots
\end{eqnarray*}

\subsubsection{Time-dependent hamiltonians}

Our next task is to characterize the symplectic micromorphisms from $\Cot\Energy\otimes\Cot Q$ to $\Cot Q$ that come from the evolution
submanifolds of (possibly time-dependent) hamiltonian
systems. We proceed by imposing obvious conditions in terms of their
local generating functions.

Let $([W],J):\Cot\Energy\otimes\Cot Q\rightarrow\Cot Q$ be a symplectic micromorphism with core map $J(Q)=(J_{E}(Q),J_{Q}(Q))$. The general form for its local generating function in an admissible local chart as in \eqref{eq: admin local chart} is
\[
	S(t,p,Q)=tJ_{E}(Q)+pJ_{Q}(Q)+f(t,p,Q),
\]
where $f$ is some function that vanishes, as well as its derivatives in the $t$ and $p$ directions, when $t=p=0$. This implies in general that
\[
   \frac{\partial^{2}S(0,0,Q)}{\partial p\,\partial Q}=\left(
                                                                                     \begin{array}{cc}
                                                                                       0                        & \nabla J_{Q}(Q)\\
                                                                                       \nabla J_{Q}(Q) & 0
                                                                                      \end{array}
                                                                               \right),
\]
because of the vanishing of $f$. Now the nondegeneracy condition \eqref{eq: non-deg for S} is satisfied iff $J_{Q}$ is a local diffeomorphism. If this is the case, Equations \eqref{eq: x_t}-\eqref{eq: p_t} define, as before, a flow $\Psi_{t}(p,q)=(P(t),Q(t))$ on $\Cot U$ such that
\[
	W_{U}=\Big\{\Big((t,\partial_{t}S(t,p,Q(t))),(p,q),\Psi_{t}(p,q)\Big):t,p\textrm{ small}\Big\}.
\]
If we also want $\Psi_{t}$ to be the identity map when $t=0$, we need to further impose that $S(0,p,Q)=pQ$, or equivalently, that
\[
	pJ_{Q}(Q)+f(0,p,Q)=pQ.
\]
Differentiating this last equation with respect to $p$ and setting $p=0$, we are left we no choice but to require that $J_{Q}(Q)=Q$. Now we can define the \textit{time-dependent} hamiltonian
\[
	H_{t}(p,q):=\partial_{t}S(t,p(t),q),
\]
where $(p(t),q(t)):=\Psi_{t}^{-1}(p,q)$. This way, the generating function satisfies by definition the time-dependent Hamilton-Jacobi equation
\[
	\partial_{t}S(t,p,Q)=H_{t}(\partial_{x}S(t,p,Q),Q)
\]
with the same initial condition for $S$ as in the time-independent case. As a consequence, $W_{U}$ is now of the form \eqref{eq: evolution submanifold} with the hamiltonian $H$ replaced with the time-dependent one $H_{t}$. The Hamilton-Jacobi theorem now tells us now that $\Psi_{t}$ is the flow of $H_{t}$ in $\Cot U$. Because our local charts are of the form
\[
	\id_{\Cot\Energy}\times\Cot\chi\times\Cot\chi:\Cot\Energy\times\Cot U\times\Cot U\longrightarrow\Cot\Energy\times\Cot Q_{|\tilde{U}}\times\Cot Q_{|\tilde{U}},
\]
where $\chi:\tilde{U}\rightarrow U$ is a local coordinate patch $\tilde{U}\subset Q$, the hamiltonian flows that are induced on the coordinate patches $\Cot Q_{|\tilde{U}}$ coincide on their overlaps. Thus this defines a time-dependent hamiltonian system $H_{t}:\Cot Q\rightarrow\R$ (where $H_{t}$ is only defined for small times and momenta) whose evolution microfold is precisely $([W],J)$. Let us sum up what we have proven so far:

\begin{prop}\label{pro: time-independent hamiltonian}A symplectic micromorphism
\[
	\rho:\Cot\Energy\otimes\Cot Q\rightarrow\Cot Q
\]
is the evolution micromorphism of a (possibly time-dependent) hamiltonian
flow if and only if
\[
	\big(\core\,\rho\big)(x)=(U(x),x)
\]
for some smooth function $U:Q\rightarrow\R$. 
\end{prop}

\subsection{Categorical mechanics}

We want to categorize the evolution micromorphisms at a purely categorical level. To give an indication as where we are aiming at, observe that the unitality condition
\[
	\rho\circ(e_{\Energy}\otimes\id_{\Cot Q})=\id_{\Cot Q}
\]
for the symplectic micromorphism $\rho:\Cot\Energy\otimes\Cot Q\rightarrow\Cot Q$ ($e_{\mathcal{E}}$ is the unique symplectic micromorphism from the cotangent microbundle of the point to $\Cot\Energy$) implies\footnote{Namely, at the level of the cores, the unitality condition imposes that
\[
	(pr_{\Energy}\times\id_{Q})\circ\core\,\rho=\id_{Q},
\]
where $\id_{Q}:Q\rightarrow Q$ is the core of $\id_{\Cot Q}$ and where the projection $pr_{\Energy}:\Energy\rightarrow\{*\}$ is the core of $e_{\Energy}$. If we denote the two components of the core map of $\rho$ by $(\Core\rho)(q)=(U(q),V(q))$, we see that this last equation is satisfied if and only if $V(q)=q$.} that $\big(\core\,\rho\big)(x)=(U(x),x)$ for some function $U:Q\rightarrow\R$. 

Therefore, thanks to Proposition \ref{pro: time-independent hamiltonian}, the unitality condition, which is purely categorical, already singles out the class of $\rho$ coming from evolution micromorphisms of  possibly \textit{time-dependent} hamiltonian systems. The time-independent case is more subtle, and requires us to look at $\Cot\Energy$ in a different way. 

First of all, consider the Lie algebra $\mathcal{T}$of the time translation group $(\R,+)$, which is the abelian Lie algebra on $\R$. We can identify $\mathcal{E}$ with the dual of $\mathcal{T}$. As the dual of a (trivial) Lie algebra, $\mathcal{E}$ can be seen as a Poisson manifold endowed with the zero Poisson structure. We call $\mathcal{T}$ the \textit{Lie algebra of time} and $\mathcal{E}$ the \textit{Poisson manifold of energy}. The cotangent microbundle $\Cot\Energy=\mathcal{T}\times\mathcal{E}$ is a symplectic groupoid (see \cite{Weinstein87}) with source and
target maps coinciding with the bundle projection; the space of composable pairs is $\Cot\Energy\oplus\Cot\Energy$, and the groupoid product is the addition of times in a fiber of constant energy:
\begin{eqnarray*}
	m_{\Energy}\big((t_{1},E),(t_{2},E)\big) & = & (t_{1}+t_{1},E).
\end{eqnarray*}
One verifies that the graph of the groupoid product is a symplectic micromorphism
\begin{eqnarray*}
	\mu_{\Energy}:=\big(\graph[m_{\Energy}],\Delta_{\mathcal{E}}):\Cot\Energy\otimes\Cot\Energy & \longrightarrow & \Cot\Energy,
\end{eqnarray*}
where the core map $\Delta_{\mathcal{E}}$ is the diagonal map on $\mathcal{E}$. It is easy to see that $\mu_{\mathcal{E}}$ further satisfies the following associativity and unitality equations
\begin{gather*}
   \mu_{\Energy}\circ(\mu_{\Energy}\otimes\id)\;=\;\mu_{\Energy}\circ(\id\otimes\mu_{\Energy}),\\
   \mu_{\Energy}\circ(e_{\Energy}\otimes\id)\;=\;\id\;=\mu_{\Energy}\circ(\id\otimes e_{\Energy}),
\end{gather*}
where $ $$e_{\Energy}$ is the unique symplectic micromorphism from the cotangent bundle of the point to $\Cot\Energy$. In other words, $(\Cot\Energy,\mu_{\Energy})$ is a monoid object in the microsymplectic category. 

\begin{prop}\label{pro:Classical mechanics}
A symplectic micromorphism $\rho:\Cot\Energy\otimes\Cot Q\rightarrow\Cot Q$ is the evolution micromorphism of a time-independent hamiltonian system if and only if 
\begin{eqnarray}
	\rho\circ(e_{\Energy}\otimes\id)     & = & \id,\label{eq: unitality}\\
	\rho\circ(\mu_{\Energy}\otimes\id) & = & \rho\circ(\id\otimes\rho),\label{eq: associativity}
\end{eqnarray}
in other words, if and only if $(\Cot Q,\rho)$ is a $\Cot\Energy$-module in the microsymplectic category.
\end{prop}

\begin{proof}
Let $(\Cot Q,\rho)$ be a $\Cot\Energy$-module. We have already seen that the unitality condition tells us that $\rho$ is the evolution micromorphism of a time-dependent hamiltonian system $H_{t}:\Cot Q\rightarrow\R$. Therefore there is a representative of the form \eqref{eq: evolution submanifold} with $H$ replaced with $H_{t}$. A direct computation with binary relations gives us a representative of the L.H.S of \eqref{eq: associativity},
\[
	\begin{gathered}
		\bigg\{\bigg(t_{2},H_{t_{1}+t_{2}}\big(\Psi_{t_{1}+t_{2}}(z)\big),t_{1},H_{t_{1}+t_{2}}\big(\Psi_{t_{1}+t_{2}}(z)\big),z,\Psi_{t_{1}+t_{2}}(z)\bigg):\, z,t_{1},t_{2}\bigg\},
	\end{gathered}
\]
as well as a representative of its R.H.S., 
\[
	\bigg\{\bigg(t_{2},H_{t_{2}}\big(\Psi_{t_{1}}\circ\Psi_{t_{2}}(z)\big),t_{1},H_{t_{1}}\big(\Psi_{t_{1}}(z)\big),z,\Psi_{t_{2}}\circ\Psi_{t_{1}}(z)\bigg):\, z,t_{1},t_{2}\bigg\}.
\]
Requiring the equality of both sides is equivalent to imposing that $\Psi_{t_{1}}\circ\Psi_{t_{2}}=\Psi_{t_{1}+t_{2}}$ and $H_{t}(z)=H_{0}(z)$ for all $t_{1},t_{2}$ and $t$. In other words, the associativity equation holds iff $\rho$ is the evolution micromorphism of a time-independent hamiltonian system. 
\end{proof}

\begin{rem}
It is straightforward to generalize the proposition above to general $\Cot\Energy$-modules
\[
	\rho:\Cot\Energy\otimes[P,Q]\rightarrow[P,Q]
\]
in the microsymplectic category by using a symplectomorphism germ $[\Psi]:\Cot Q\rightarrow[P,Q]$ coming from the lagrangian embedding theorem.
\end{rem}

\subsection{Classical symmetries}

It is possible to generalize the previous scheme to a general Hamiltonian action of a Lie group $G$ on $\Cot Q$ with momentum map $j:\Cot Q\rightarrow\mathcal{G}^{*}$. In this case, we define the symmetry submanifold to be
\[
	W_{G} :=  \bigg\{\Big(\big(v,j(\exp(v)z)\big),z,\exp(v)z)\Big):\; v\in U,\, z\in\Cot Q\bigg\},
\]
where $U\subset \mathcal{G}$ is the maximal neighborhood of $0$ in the Lie algebra $\mathcal{G}$ on which the exponential mapping $\exp:U \rightarrow G$ is a diffeomorphism on its image. Taking the germ of $W_{G}$ around the graph of $j_{|Q}\times\id_{Q}$, yields a symplectic micromorphism 
\[
	\rho_{G}:\Cot\mathcal{G}^{*}\otimes\Cot Q\longrightarrow\Cot Q
\]
Now, thanks to the exponential mapping, we can define a generating function germ from $\Cot\mathcal{G}^{*}\oplus\Cot\mathcal{G}^{*}$ to $\R$ via the formula
\[
	S_{G}(v,w,\mu)  :=  \Big\langle\mu,\exp^{-1}\big(\exp(v)\exp(w)\big)\Big\rangle,
\]
where $\langle\,,\,\rangle$ is the canonical paring between the Lie algebra and its dual. This generating function germ defines a symplectic micromorphism $\mu_{G}$ from $\Cot\mathcal{G}^{*}\otimes\Cot\mathcal{G}^{*}$ to $\Cot\mathcal{G}^{*}$. One can show that $(\Cot\mathcal{G}^{*},\mu_{G})$ is a monoid and that $(\Cot Q,\rho_{G})$ a $\Cot\mathcal{G}^{*}$-module. This situation will be treated in full details elsewhere. 

As a final comment, let us mention that our approach to hamiltonian flows and symmetries through symplectic micromorphisms is close in spirit to the work of Almeida and Rios \cite{RA2004}, Benenti \cite{Benenti1983} and Zakrzewski \cite{Zakrzewski88} on generalized versions of the Hamilton-Jacobi theory for classical symmetries.

\end{document}